\documentclass [11pt]{article}

\usepackage{hyperref}
\usepackage{amsmath, amssymb, amsfonts, amsthm, enumerate}
\date{}

\newtheorem{theorem}{\bf Theorem}[section]
\newtheorem{claim}[theorem]{\bf Claim}
\newtheorem{lemma}[theorem]{\bf Lemma}
\newtheorem{conjecture}[theorem]{\bf Conjecture}

\newtheorem{corollary}[theorem]{\bf Corollary}

\title{\vspace{-1.2cm} Ramsey goodness of paths}

\author{Alexey Pokrovskiy\thanks{Department of Mathematics, ETH, 8092 Zurich, Switzerland. {\tt dr.alexey.pokrovskiy@gmail.com}. Research supported in part by SNSF grant 200021-149111.}
\and
Benny Sudakov
\thanks{Department of Mathematics, ETH, 8092 Zurich, Switzerland. {\tt benjamin.sudakov@math.ethz.ch}. Research supported in part by SNSF grant 200021-149111.}
}

\begin{document}
\maketitle

\begin{abstract}
Given a pair of graphs $G$ and $H$, the Ramsey number $R(G,H)$ is the smallest $N$ such that every red-blue coloring of the edges of the complete graph $K_N$ contains a red copy of $G$ or a blue copy of $H$. If graph $G$ is connected, it is well known and easy to show that $R(G,H) \geq (|G|-1)(\chi(H)-1)+\sigma(H)$, where $\chi(H)$ is the chromatic number of $H$ and $\sigma$ the size of the smallest color class in a $\chi(H)$-coloring of $H$. A  graph $G$ is called $H$-good if $R(G,H)= (|G|-1)(\chi(H)-1)+\sigma(H)$. The notion of Ramsey goodness was introduced by Burr and Erd\H{o}s in 1983 and has been extensively studied since then. In this short note we prove that $n$-vertex path $P_n$ is $H$-good for all $n\geq 4|H|$. This proves in a strong form a conjecture of Allen, Brightwell, and Skokan.
\end{abstract}

\maketitle

\section{Introduction}
Given a pair of graphs $G$ and $H$, the Ramsey number $R(G,H)$ is the smallest $N$ such that every red-blue coloring of the edges of the complete graph $K_N$ contains a red copy of $G$ or a blue copy of $H$. It is a corollary of the celebrated theorem of Ramsey that these numbers are always finite. Let $\chi(H)$ be the chromatic number of $H$, i.e. the smallest number of colors needed to color the vertices of $H$ so that no pair of adjacent vertices have the same colour, and $\sigma(H)$ be the the size of the smallest color class in a $\chi(H)$-colouring of $H$.
It was observed by Burr \cite{Bu} that for connected $G$ with $|G|\geq\sigma(H)$ Ramsey numbers always satisfy the following easy lower bound
\begin{equation}\label{RamseyLowerBound}
R(G,H)\geq  (|G|-1)(\chi(H)-1)+\sigma(H).
\end{equation}
To prove (\ref{RamseyLowerBound}), consider a $2$-edge-coloring of the complete graph on  $N=(|G|-1)(\chi(H)-1)+\sigma(H)-1$ vertices consisting of $\chi(H)-1$ disjoint red cliques of size $|G|-1$ as well as one disjoint red clique of size $\sigma(H)-1$. This coloring has no red $G$ because all red connected components have size $\leq |G|-1$, and there is no blue $H$ since the partition of this $H$ induced by red cliques would give a coloring of $H$ by $\chi(H)$ colors with one color class smaller than $\sigma(H)$, contradicting the definition of $\sigma(H)$.

For some graphs the bound in (\ref{RamseyLowerBound}) is quite far from the truth.  For example Erd\H{o}s~\cite{ErdosRamsey} showed that $R(K_n, K_n)\geq \Omega(2^{n/2})$ which is much larger than the quadratic bound we get from (\ref{RamseyLowerBound}). However there are many known pairs of graphs for which $R(G,H)=(|G|-1)(\chi(H)-1)+\sigma(H)$. In this case we say that \emph{$G$ is $H$-good}. The notion of Ramsey goodness was introduced by Burr and Erd\H{o}s \cite{BE} in 1983 and was extensively studied since then, see, e.g., \cite{ABS,Ch, CFLS, FGMSS,Nik, NR} and their references.

In this short note we study the question of when the $n$-vertex path $P_n$ is $H$-good, for some fixed graph $H$. This problem goes back to the work of Erd\H{o}s, Faudree, Rousseau, and Schelp~\cite{EFRSMultipartiteSparse}, who in 1985 proved that there is a function $f$ such that 
$P_n$ is $H$-good for all $n\geq f(|H|)$. The function $f(|H|)$ is not explicit in~\cite{EFRSMultipartiteSparse}, but $f(H)=O(|H|^4)$ can be proved using their method. H\"aggkvist \cite{H} (for $k=2$) and later Pokrovskiy \cite{P} obtained a general upper bound on the Ramsey number of path versus complete $k$-partite graphs, showing that $R(P_n, K_{m, \ldots, m}) \leq (k-1)(n-1)+km-k+1$. Here and later, 
$K_{m_1, \ldots, m_k}$ denotes a complete $k$-partite graph with parts of order $m_1, \ldots, m_k$ together with all the edges connecting vertices in different parts. Although this bound is not strong enough to prove goodness, it holds for all values of the parameters. More recently, Pei and Li~\cite{PeiLi} showed that if $n\geq 8|H|+3\sigma(H)^2+c\chi^8(H)$, then $P_n$ is $H$-good. For general $H$ (e.g., when $H=K_{m,m}$) this result requires $n$ to be quadratic in $|H|$.  Allen, Brightwell, and Skokan~\cite{ABS} conjectured that $P_n$ is $H$-good already when $n$ is linear in $|H|$.
\begin{conjecture}[\cite{ABS}]~\label{ConjectureABS}
Let $H$ be a fixed graph with chromatic number $k$ and let $n\geq k|H|$. Then
$R(P_n,H)=(n-1)(k-1)+\sigma(H)$. 
\end{conjecture}
 
Let $R(C_{\geq n}, H)$ be the smallest $N$ such that any $2$-edge-coloring of $K_N$ contains either a red cycle of length at least $n$ or a blue $H$. Notice that we always have $R(P_n, H)\leq R(C_{\geq n}, H)$. Motivated by the above conjecture, in this note we prove the following theorem.

\begin{theorem}\label{RamseyAtLeastTheorem}
Given integers $m_1\leq m_2\leq \dots \leq m_k$ and $n\geq 3m_k+5m_{k-1}$, we have
$$R(C_{\geq n}, K_{m_1, \dots, m_k})= (k-1)(n-1)+m_1.$$
\end{theorem}

\noindent
Notice that the vertices of a $k$-chromatic graph  $H$ can be partitioned into $k$ independent sets of sizes $m_1, \dots, m_k$ with  
$\sigma(H)=m_1\leq m_2\leq \dots \leq m_k$.  This is equivalent to $H$ being a subgraph of $K_{m_1, \dots, m_k}$.
Since $4|H|\geq 4m_k+4m_{k-1}\geq 3m_k+5m_{k-1}$, Theorem~\ref{RamseyAtLeastTheorem} implies the following.

\begin{corollary}\label{RamseyPathTheorem}
Let $H$ be a fixed graph with chromatic number $k$ and let $n\geq 4|H|$. Then
$R(P_n,H)=(n-1)(k-1)+\sigma(H)$.
\end{corollary}

For $k\geq 4$, this corollary proves Conjecture~\ref{ConjectureABS} in a very strong form, showing that the condition $n\geq \chi(H)|H|$ is unnecessary, and $n\geq 4|H|$ suffices.  For $k\leq 3$, our result is slightly weaker than the conjecture, but is a large improvement on the best previously known~\cite{PeiLi} quadratic dependence of $n$ on $|H|$. Moreover, for certain graphs $H$, Theorem~\ref{RamseyAtLeastTheorem} shows that $P_n$ is $H$-good even when $n$ is smaller than $4|H|$. For example if $H$ is balanced (i.e. if $|H|=\sigma(H)\chi(H)$), then this theorem implies that $P_n$ is $H$-good as long as $n\geq 8|H|/\chi(H)$.

\section{Proof of the main theorem}
The proof of Theorem~\ref{RamseyAtLeastTheorem} uses a combination of the P\'osa rotation-extension technique and induction on $k$.
Let $P=p_1 p_2 \dots p_t$ be a path in a graph $G$ which also contains an edge $p_t p_i$. We say that a path 
$Q= p_1 p_2 \dots p_i p_t p_{t-1} \dots p_{i+1}$ is a rotation of $P$. We say that a path $Q$ is \emph{derived} from $P$ if there is a sequence of paths $P_0=P, P_1, \dots, P_s=Q$ with $P_j$ being a rotation of $P_{j-1}$ for each $j$.
We say that a vertex $x$ is an ending vertex for $P$ if it is the final vertex of some path derived from $P$.  
We say that a set  $X$ of ending vertices of $P$ is \emph{connected} if for every $x\in X$ there is a path $P_x$ ending with $x$ and a sequence of paths 
$P=P_0, P_1, \dots, P_s= P_x$ with  $P_j$ being a rotation of $P_{j-1}$ for each $j$ and $P_j$ ending with some $x'\in X$. 
That is, for a $P_x$ ending in $x \in X$, the set $X$ also contains ending vertices of a sequence of intermediate paths whose rotations produce $P_x$.   
Notice that for a path $P$, the set of all ending vertices for $P$ is connected.
For a set of vertices $S$ in a graph $G$, we use $N_G(S)$ to mean the set of vertices  $v\in V(G)\setminus S$ for which there is a vertex $s\in S$ such that $vs$ is an edge of $G$ i.e $N_G(S)$ is the neighbourhood of the set $S$ outside $S$. When there is no ambiguity of what the underlying graph is, we abbreviate this to $N(S)$.

We will need a variant of the celebrated lemma of P\'osa from~\cite{Pos}. 
\begin{lemma}\label{LemmaEndpointsNeighbourhoodContained}
Let $P$ be a path of maximum length in a graph $G$ starting at some vertex $p_0$, $X$ a connected set of ending vertices for $P$, and $S$ the set of all ending vertices for $P$.
Then we have 
\begin{enumerate}[(i)]
\item $|N(X)|\leq 2|S|.$
\item There is a cycle in $G$ containing  $N(X)\cup X$.
\end{enumerate}
\end{lemma}
\begin{proof}
Fix an orientation of $P$.
Let $S^-$ be the set of left neighbours on $P$ of vertices in $S$ and let $S^+$ be the set of right neighbours on $P$ of vertices in $S$. Notice that by maximality of $|P|$,  we have $N(X)\subseteq V(P)$ (indeed otherwise if we have $y\in N(X)\setminus V(P)$, then we can obtain a longer path starting from $p_0$ by choosing a rotation of $P$ ending with a neighbour of $y$, and extending it with $y$.)
First we will show that $N(X)\cup X\subseteq S^-\cup S^+\cup X$. 

Let $x$ be a vertex in $X$ and $P=P_0,P_1 \dots, P_s$ be a sequence of paths such that $P_s$ ends with $x$, $P_i$ ends with some $v\in X$, and $P_i$ is a rotation of $P_{i-1}$ (such a sequence exists since $X$ is a connected set of endpoints.) Notice that if $yz\in E(P_s)$, then either $yz\in E(P)$ or $\{y,z\}\subseteq X\cup X^-\cup X^+$ (this is proved by induction using the fact that from the definition of ``rotation'', we have $yz\in P_i \implies yz\in P_{i-1}$ or $\{y,z\}\subseteq X\cup X^-\cup X^+$.)

Let $p_0, \dots, p_k=x$ be the vertex sequence of $P_s$. For $p_i\in N(x)$, we know $p_1 p_2 \dots p_i p_k p_{k-1} \dots p_{i+1}$ is a rotation of $P$ and hence $p_{i+1}\in S$. Since $p_i p_{i+1}\in E(P_s)$, we have that either $p_ip_{i+1}\in E(P)$ which implies that $p_i\in S^-\cup S^+\subseteq  X\cup S^-\cup S^+$, or we have $\{p_{i}, p_{i+1}\}\subseteq X\cup X^-\cup X^+\subseteq X\cup S^-\cup S^+$. This implies that $N(X)\subseteq\bigcup_{x\in X}N(x)\subseteq X\cup S^-\cup S^+$ as required.
Now part (i)  of the lemma comes from $|N(X)\cup X|\leq |S^-|+|S^+|+|X|\leq 2|S|+|X|$.

Although  part (ii) of the lemma already appeared in \cite{BBDK}, we include its short proof for the sake of completeness. Let $p_0, \dots, p_k$ be the vertex sequence of $P$, and let $p_i$ be the first element of $(N(X)\cup X)\cap\{p_0, \dots, p_k\}$. Since we are considering paths starting from $p_0$, we must have $p_0\not\in X$. If $p_0\in N(X)$, then we have a cycle on $V(P)$ formed by taking a rotation of $P$ ending with $p\in N(p_0)$ and adding the edge  $p_0 p$. Therefore, we can assume that $p_i\neq p_0$.
Notice that since $p_{i-1}\in N(p_i)$ we must have  $p_i\not \in X$ (otherwise $p_{i-1}$ would be an element of $(N(X)\cup X)\cap\{p_0, \dots, p_k\}$ preceding $p_i$) which implies that $p_i\in N(X)$. Let $x\in X$ be a neighbour of $p_i$, and let $Q$  be a rotation of $P$ ending with $x$. Let $p_0 q_1 \dots q_k$ be the vertex sequence of $Q$. Recall that for all $i$ either $q_t q_{t+1}\in E(P)$ or $\{q_t, q_{t+1}\}\subseteq  X^-\cup X^+\cup X$. Since $Q$ starts with $p_0$, and $X\cap \{p_0, \dots, p_{i}\}=\emptyset$, each of the edges $p_0p_1, \dots, p_{i-1}p_{i}$ must be edges of $Q$, and so $Q$ starts with the sequence $p_0p_1\dots p_{i}$. Therefore letting $C$ be the cycle formed from $Q\setminus\{p_0,\dots, p_{i-1}\}$ by adding the edge $x p_i$ gives a cycle containing $X\cup N(X)$.
\end{proof}

We now prove the main result of this note.
\begin{proof}[Proof of Theorem~\ref{RamseyAtLeastTheorem}]
The lower bound of the theorem follows from (\ref{RamseyLowerBound}), so it remains to show that $R(C_{\geq n}, K_{m_1, \dots, m_k})\leq (k-1)(n-1)+m_1.$ The proof is by induction on $k$. In case $k=1$ we need to find either a red copy of $C_{\geq n}$ or a blue $1$-partite graph $K_m^{1}$ on $m$ vertices. Since such a graph contains no edges its copy exist in any $2$-edge-coloring of $K_m$. This implies that $R(H,K_m^1)\leq m$ for any graph~$H$.
 
Now let $k\geq 2$ and suppose that the theorem holds for all $k'<k$. Let $\Gamma$ be a $2$-edge-colored complete graph on $(k-1)(n-1)+m_1$ vertices. Suppose that $\Gamma$ contains no blue $K_{m_1, \dots, m_k}$ and let $G$ be the subgraph spanned by the red edges of $\Gamma$. 
\begin{claim}\label{AtLeastExpansionInduction}
For any set $B$ with $|B| \geq m_1$, we have $|N_G(B)\cup B|\geq n-m_2+m_1$. 
\end{claim}
\begin{proof}
 Suppose that $|N_G(B)\cup B|\leq  n-m_2+m_1-1$. Let $\Gamma'=\Gamma\setminus (N_G(B)\cup B)$. We have $|\Gamma'|=|\Gamma|-|N_G(B)\cup B|\geq (k-2)(n-1)+m_2$. By induction $\Gamma'$ either contains a red $C_{\geq n}$ or a blue $K_{m_2,  \dots, m_{k}}$. In the second case, since $|B|\geq m_{1}$, and all the edges between $B$ and $\Gamma'$ are blue, $B$ can be joined to $K_{m_2,  \dots, m_{k}}$ in order to produce a blue $K_{m_1,  \dots, m_k}$.
\end{proof}

Let $A$ be a maximum size subset of $G$ satisfying $|A|< 2m_1$ and $|N_G(A)|\leq 2|A|$. 
This implies that $|N_G(A)\cup A|\leq 3|A|< 6m_1< n-m_2+m_1$, and so by Claim~\ref{AtLeastExpansionInduction} we have $|A|<m_{1}$.
Let $G'$ be the induced subgraph of $G$ on $V(G)\setminus A$. Note that every subset $X$ with $|X| \leq m_1$ in $G'$ satisfies that
$|N_{G'}(X)| > 2|X|$---indeed otherwise we would have 
$$|N_G(X\cup A)|\leq |N_G(X)\setminus A|+|N_G(A)|=|N_{G'}(X)|+|N_G(A)|\leq 2|X|+2|A|=2|A\cup X|.$$
This would contradict the maximality of $A$.

Let $P$ be a maximum length path in $G'$, and $S$ the set of ending vertices for $P$.
By Lemma~\ref{LemmaEndpointsNeighbourhoodContained} we have $|N_{G'}(S)|\leq 2|S|$ which implies that $|S|\geq m_1$.
Let $X$ be a connected set of exactly $m_1$ ending vertices for $P$. This set can be obtained by repeatedly applying rotations, adding one new ending vertex at a time until we have $m_1$ ending vertices. This guarantees that the resulting set is connected, since we always keep all previous ending vertices.
From Claim~\ref{AtLeastExpansionInduction} we obtain 
$$|N_{G'}(X)\cup X| \geq  n-m_2+m_1-|A|\geq 2m_2+5m_1.$$%
Since $|X|=m_1$ and $|N_{G'}(X)\cup X|=|N_{G'}(X)|+|X|$, the above is equivalent to $|N_{G'}(X)|\geq 2m_2+4m_1$.  Combining this with Lemma~\ref{LemmaEndpointsNeighbourhoodContained} gives
$$|S|\geq \frac{1}{2}|N_{G'}(X)|\geq m_2+2m_1.$$
Since $A$ has at most $2|A|\leq 2m_1$ neighbours in $G$, we can choose $S'\subseteq S$ with $|S'|=m_2$ such that there are no edges between $S'$ and $A$ in $G$. Note that then $N_{G'}(S')= N_G(S')$.

By Lemma~\ref{LemmaEndpointsNeighbourhoodContained}, there is a cycle $C$ in $G'$ containing $S'\cup N_{G'}(S')=S'\cup N_G(S')$. If $|C|\geq n$, then the complete graph $\Gamma$ contains a red cycle of length at least $n$. Otherwise, if $|C|\leq n-1$, then let $\Gamma'=\Gamma\setminus (S'\cup N_G(S'))$. We have $|\Gamma'|\geq |\Gamma|-|C|\geq (k-2)(n-1)+m_1$. Therefore, by induction, $\Gamma'$ contains a blue $K_{m_1, m_3, \dots, m_{k}}$. All the edges between $S'$ and $\Gamma'$ are blue, and so $S'$ can be joined to $K_{m_1, m_3, \dots, m_{k}}$ in order to produce a blue $K_{m_1,  \dots, m_k}$.
\end{proof}

\section{Concluding remarks}
An interesting open question is to determine all $n$ for which $P_n$ is $H$-good. 
It is possible to show that
$P_n$ is not $K_{m_1, m_2}$-good when $n \leq 2m_2-2$. 
To see this, let $n=2m_2-2$, $N=n+m_1-1$ and consider the edge coloring of the complete graph $K_N$, consisting of two blue cliques of orders $m_2+m_1-1$ and $m_2-2$ respectively with all the edges between them red. It is easy to see that this coloring contains no red path $P_n$ and no blue 
$K_{m_1, m_2}$.
This implies that constant ``$4$'' in Corollary \ref{RamseyPathTheorem} cannot be made less than $2$.

Our main result gives either a red cycle of length at least $n$ or a blue copy of $H$.
One could ask whether under the same conditions, we can guarantee a cycle of length \emph{exactly} $n$. Indeed, this was asked by Allen, Brightwell, and Skokan~\cite{ABS}, who conjectured that 
$R(C_n, H)=(n-1)(\chi(H)-1)+\sigma(H)$ when $n\geq \chi(H)|H|$. In a forthcoming paper \cite{PS}, we will prove a strengthening of this conjecture for large $\chi(H)$ and $\sigma(H)$. Our result shows that $C_n$ is $H$-good when
$n \geq C |H|$ for some constant $C$ and $\sigma(H) \geq \chi(H)^{11}$.

We also like to mention a 
related problem of Erd\H{o}s, Faudree, Rousseau, and Schelp \cite{EFRSCycleComplete}, who conjectured  that $R(C_n, K_m)=(n-1)(m-1)+1$ 
for $n\geq m$. The best known bound for this conjecture is due to Nikiforov \cite{Nik}, who proved it when $n\geq 4m+1$.

Since paths are special case of trees, it is natural to consider the Ramsey goodness problem for trees as well. For example an old result of Chv\'atal \cite{Ch} says that any tree $T$ is good for every complete graph. Motivated by question of
Erd\H{o}s, Faudree, Rousseau and Schelp, in the recent paper \cite{BPS} we study the Ramsey goodness of bounded degree trees with respect to general graphs.

\vspace{0.15cm}
\noindent
{\bf Acknowledgment.}\, 
Part of this work was done when the second author visited Freie University Berlin.
He would like to thank Humboldt Foundation for a generous support during this visit and 
Freie University for its hospitality and stimulating research environment. Both authors would like to thank J. Skokan, M. Stein and D. Hefetz 
for helpful discussions, as well as D. Korandi for finding a mistake in an earlier draft of this note.

\end{document}